\documentclass{amsart}

\usepackage{graphicx}

\usepackage{amsmath}

\usepackage{amsfonts}

\usepackage{amssymb}

\newtheorem{theorem}{Theorem} [section]

\newtheorem{thm}[theorem]{Theorem}

\newtheorem{cor}[theorem]{Corollary}

\newtheorem{lemma}[theorem]{Lemma}

\newtheorem{lma}[theorem]{Lemma}

\newtheorem{prop}[theorem]{Proposition}

\begin{document}

\title{Intrinsically Knotted and 4-Linked Directed Graphs}

\author{Thomas Fleming}
\author{Joel Foisy}

\begin{abstract}
We consider intrinsic linking and knotting in the context of directed graphs.  We construct an example of a directed graph that contains a consistently oriented knotted cycle in every embedding.   We also construct examples of intrinsically 3-linked and 4-linked directed graphs.  We introduce two operations, consistent edge contraction and H-cyclic subcontraction, as special cases of minors for digraphs, and show that the property of having a linkless embedding is closed under these operations.  We analyze the relationship between the number of distinct knots and links in an undirected graph $G$ and its corresponding symmetric digraph $\overline{DG}$.  Finally, we note that the maximum number of edges for a graph that is not intrinsically linked is $O(n)$ in the undirected case, but $O(n^2)$ for directed graphs.
\end{abstract}

\maketitle

\section{Introduction}

An abstract graph G is called \emph{intrinsically linked} (\emph{intrinsically
knotted}) if every embedding of that graph into three-space contains
a nontrivial link (knot). Intrinsically linked and intrinsically
knotted graphs were first studied in the early 1980s by Sachs
\cite{sachs} and by Conway and Gordon \cite{cg}.  In this paper, we will consider directed graphs and use the terms digraph and directed graph interchangeably.  Undirected graphs will be denoted $G$, and directed graphs as $\overline{G}$.  For a graph $G$, the digraph $\overline{DG}$ is the symmetric digraph corresponding to $G$ formed by replacing each edge $v_1 v_2$ of $G$ with the directed edges $v_1v_2$ and $v_2v_1$.

In \cite{FHR}, Foisy, Howards and Rich consider intrinsic linking in the context of directed graphs, defining a directed graph $\overline{G}$ to be intrinsically linked (knotted) if every embedding of the digraph $\overline{G}$ contains a nontrivial link (knot) where every cycle in the link (knot) has a consistent orientation. This definition is motivated by applications where the graph represents a flow, such as an electrical current, and the edges of the graph must have a consistent orientation for the flow to occur.  Foisy, Howards and Rich then prove that the complete symmetric digraph on six vertices $\overline{DK_6}$ is intrinsically linked. We continue their work by finding an example of an intrinsically knotted directed graph in Section \ref{i_knot_section}.

Variations of intrinsic linking have been studied that require more complex structures in every embedding of the graph, such as a nonsplit link of $n$ components \cite{flapan2}, a link where one or more components are non-trival knots \cite{flem}, or arbitrary linking patterns between the components \cite{flappan}.  We begin extending the theory of intrinsic linking for directed graphs in this direction by demonstrating examples of intrinsically 3- and 4-linked directed graphs in Sections \ref{3_link_section} and \ref{4_link_section}.

Researchers have also studied the minimal number of links (knots) that are contained in the embeddings of an intrinsically linked (knotted) graph \cite{fmellor}, \cite{mellor2}, see also \cite{survey} for a summary of similar results when restricting to straight line embeddings of such graphs.  Inspired by a question of \cite{FHR}, we investigate the number of distinct knots or links in directed graphs of the form $\overline{DG}$.  We construct bounds on the minimum number of knots or links in $\overline{DG}$ based on the minimum number of knots or links in $G$, and show that if $\overline{DG}$ is intrinsically linked, then it must contain at least 4 distinct, consistently oriented links.   This is in contrast to arbitrary directed graphs, and we provide an example of an intrinsically linked directed graph that admits an embedding with a single non-trivial consistently oriented link.

For undirected graphs, the property of having a linkless or knotless embedding is closed under the operation of taking graph minors \cite{nt}, and the minor minimal family for intrinsically linked graphs was characterized by Robertson, Seymour and Thomas \cite{rst}.  In a sharp contrast, Foisy, Howards and Rich \cite{FHR} show that this is not the case for intrinsic linking in directed graphs-- a directed graph that has an embedding with no non-split consistently oriented link may have a minor that is intrinsically linked.  In Section \ref{consistent_section}, we examine the operation of vertex expansion, and answer a question of \cite{FHR}.  We then introduce the notions \emph{consistent edge contraction} and \emph{H-cyclic subcontraction}.  Consistent edge contraction preserves the property of having a linkless embedding for digraphs, and H-cyclic subcontraction does as well under mild assumptions. Consistent edge contraction is used for some of our proofs, but further investigation is needed to determine the best analogue of minors for studying intrinsically linked directed graphs, and whether any of the possible operations will lead to a finite forbidden set for intrinsically linked directed graphs.

Finally, in Section \ref{density_section} we note an additional difference between intrinsic linking in the directed and undirected cases.  We consider the extremal problem of the maximum number of edges for a graph $G$ or a digraph $\overline{G}$ on $n$ vertices to admit a linkless embedding. For undirected graphs, there is a constant $c$ such that any graph on $n$ vertices with more than $cn$  edges contains a $K_6$ minor \cite{mader} \cite{thom}, hence is intrinsically linked. However for undirected graphs, as noted in \cite{FHR}, there are examples of directed graphs with $O(n^2)$ edges that are not intrinsically linked.  Thus, as $n$ becomes large, we should expect many directed graphs to be intrinsically linked as undirected graphs (that is, contain a non-split link in every embedding when ignoring edge orientations), but not be intrinsically linked as directed graphs (i.e. admit an embedding where all consistently oriented cycles form only split links). A similar result holds for intrinsic knotting.

\section{Vertex expansion and graph minors}
\label{consistent_section}

We first consider vertex expansion in digraphs, and then turn our attention to operations that are variations of edge contraction.  In \cite{FHR}, Question 5.1 asks if conducting a single vertex expansion of an intrinsically linked directed graph preserves the property of intrinsic linking for some choice of orientation for the new edge. We demonstrate the answer to this question is ``no'' by providing a counterexample.

\begin{thm}
Let $\overline{G}$ be the directed graph in Figure \ref{counterexample}, and $\overline{G'}$ and $\overline{G''}$ the directed graphs obtained by vertex expansion at $a_2$ as shown in Figure \ref{counterexpand}.  Then $\overline{G}$ contains a consistently oriented non-split link in every spatial embedding, but both $\overline{G'}$ and $\overline{G''}$ have embeddings with no non-split consistently oriented link.
\label{vertex_exp_thm}
\end{thm}

\begin{proof}

We can see that the directed graph $\overline{G}$ is intrinsically linked as follows. For a given embedding of $\overline{G}$, when ignoring the edges from vertex set $B$ to $A$, we may find a triangle $T$  and a square $S$ with odd linking number, as $K_{3,3,1}$ is a minor of this graph.   Note that by construction, $T$ is consistently oriented.  If $S$ is not consistently oriented, we may construct a consistently oriented cycle that has odd linking number with $T$ using the edges directed from $B$ to $A$ and the techniques described in the proof of Theorem 3.5 of \cite{FHR}.  If $S$ contains the edge $a_2b_2$, this merely forces the orientation for the construction.

We now show that $\overline{G'}$ and $\overline{G''}$ admit linkless embeddings.  For simplicity, it will suffice to consider embeddings where all bigons bound disks.  There is an embedding of $K_{3,3,1}$ that contains a single nontrivial link, consisting of a triangle and square \cite{fmellor}, and by symmetry, every triangle-square partition is the lone non-trivial link in some embedding.  Let $f$ be an embedding of $K_{3,3,1}$ where the sole non-trivial link is $(x,a_2,b_3) (a_1,b_1,a_3,b_2)$.  This gives an embedding of $\overline{G}$, which we will call $f$ by abuse of notation, where the only nontrivial consistently oriented links are $(x,a_2,b_3) (a_1,b_1,a_3,b_2)$ and $(x,a_2,b_3) (a_1,b_2,a_3,b_1)$.  Let $f'$ be an embedding of $K_{3,3,1}$ where the sole non-trivial link is $(x,a_1,b_1)(a_2,b_2,a_3,b_3)$.  Due to the single edge from $a_2$ to $b_2$ in $\overline{G}$, the embedding $f'$ of $\overline{G}$ contains a single consistently oriented non-trivial link $(x,a_1,b_1) (a_2,b_2,a_3,b_3)$.

\begin{figure}
\includegraphics[scale=.5]{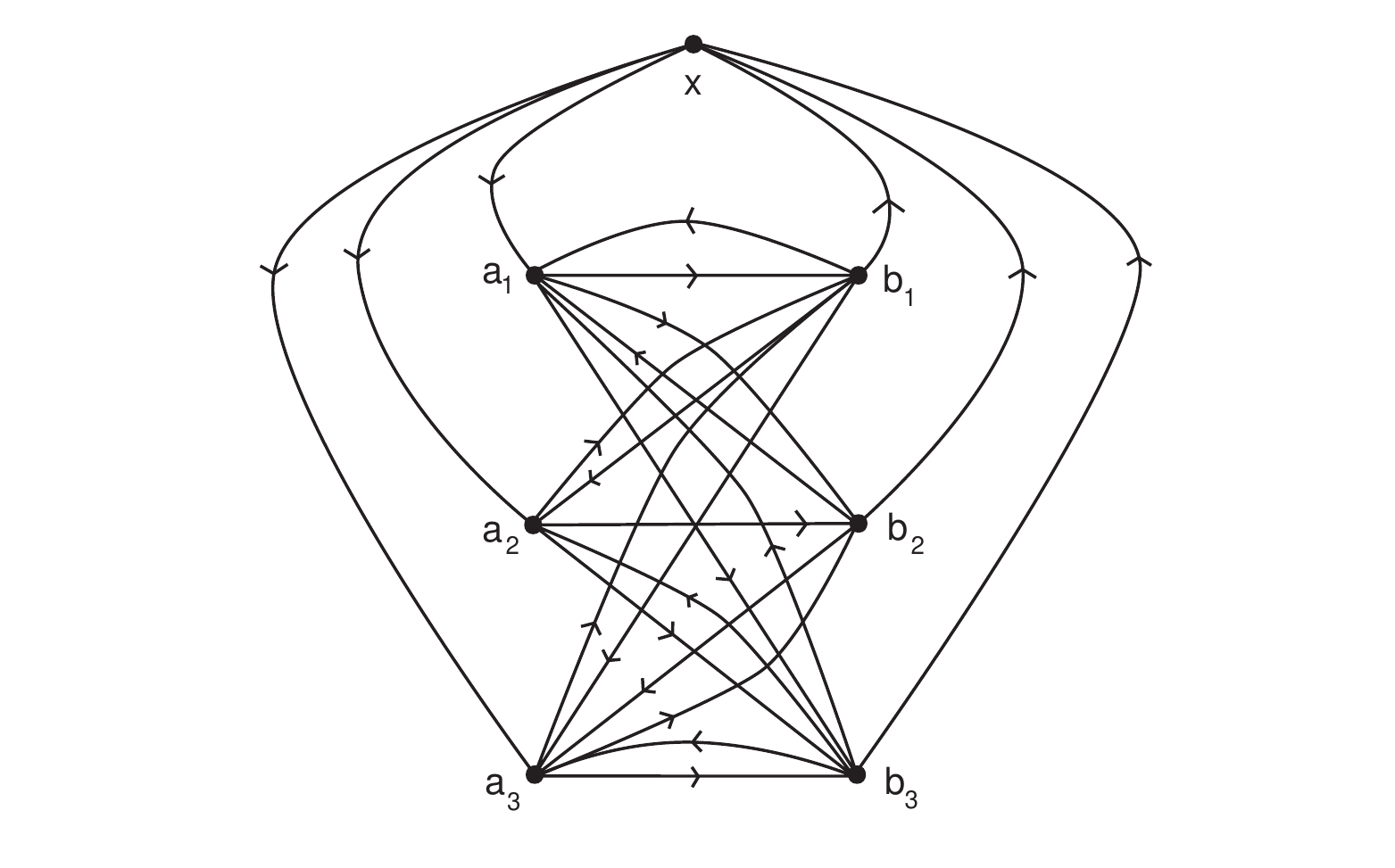}
\caption{The digraph $\overline{G}$ for Theorem \ref{vertex_exp_thm}. Note the single edge from $a_2$ to $b_2$.}
\label{counterexample}
\end{figure}

\begin{figure}
\includegraphics[scale=.5]{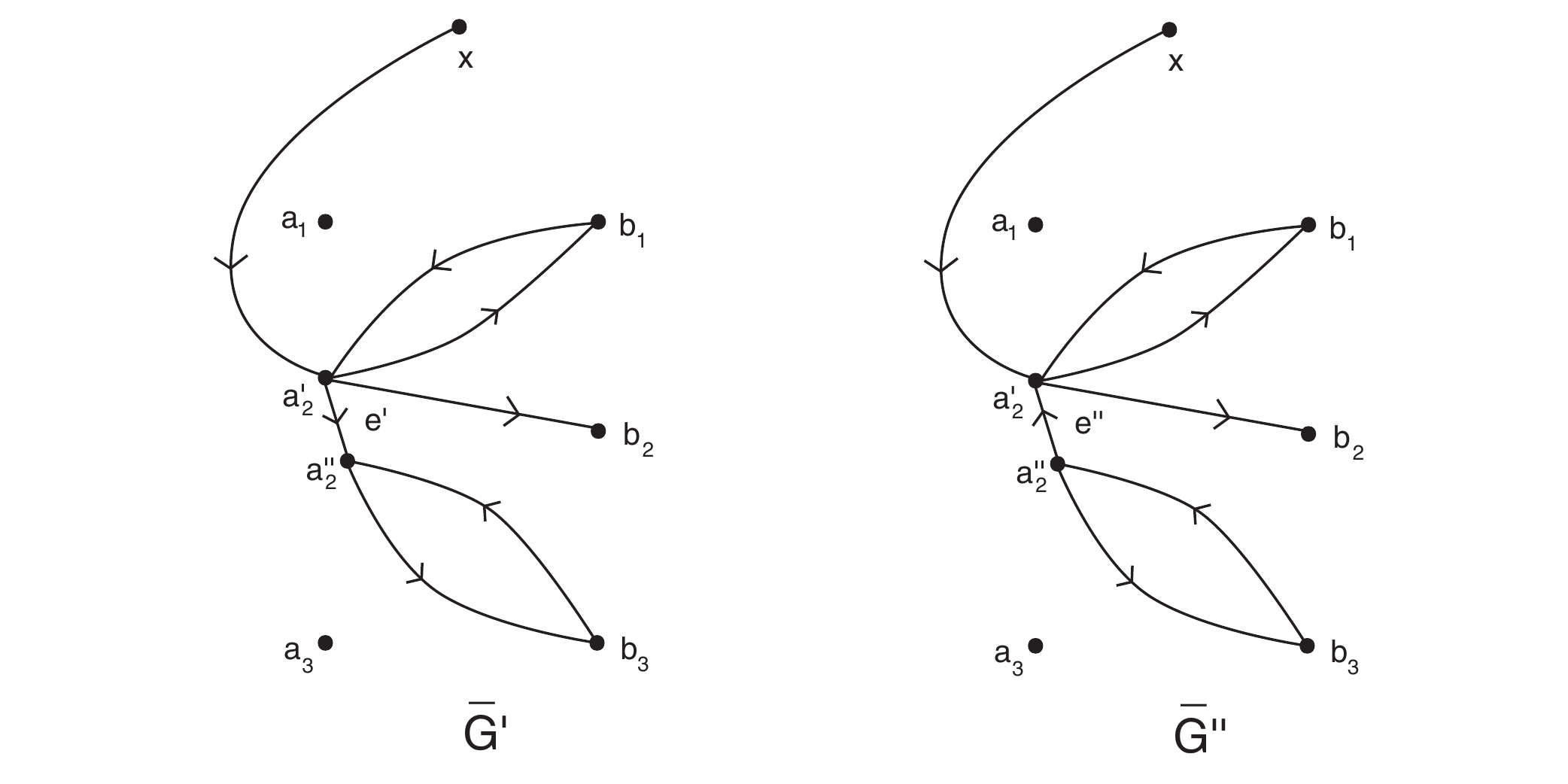}
\caption{The edges adjacent to $a'_2$ and $a''_2$ in $\overline{G'}$ and $\overline{G''}$.}
\label{counterexpand}
\end{figure}

The graphs $\overline{G'}$ and $\overline{G''}$ are formed from $\overline{G}$ by expansion of vertex $a_2$.  By shrinking the edge $e'$ or $e''$, we may cause vertices $a'_2$ and $a''_2$ and edge $e'$ (or $e''$) to lie in a neighborhood of vertex $a_2$, so we may extend the embeddings $f$ and $f'$ to $\overline{G'}$ and $\overline{G''}$.  Now for the digraph $\overline{G'}$, the cycle $(a''_2,a'_2,b_2,a_3,b_3, )$ is inconsistently oriented.  Suppose that the embedding $f'$ of $\overline{G'}$ contained a consistently oriented non-split link that did not use the edge $e'$, then embedding $f'$ of $\overline{G}$ would contain more than one consistently oriented non-split link, a contradiction.  If embedding $f'$ of $\overline{G'}$ contains a consistently oriented non-split link that does contain $e'$, then that link is preserved by contracting $e'$, again creating a second, distinct non-split link in embedding $f'$ of $\overline{G}$. Thus, the embedding $f'$ of $\overline{G'}$ has no consistently oriented link.   Similarly, for $\overline{G''}$, the cycle $(x,a'_2,a''_2,b_3)$ is inconsistently oriented, so the embedding $f$ of $\overline{G''}$ has no consistently oriented link.

\end{proof}

Foisy-Howard-Rich noted in Theorem 4.2 of \cite{FHR} that standard vertex expansion does not preserve intrinsic linking in directed graphs, and Theorem \ref{vertex_exp_thm} shows that controlling the edge orientation in arbitrary vertex expansion is also insufficient.  We introduce two operations to address these challenges: \emph{consistent edge contraction} and \emph{H-cyclic subcontraction}. We describe these each in turn.

Let $\overline{G'}$ be the digraph obtained from $\overline{G}$ by splitting a vertex $v$ into vertices $v_1$ and $v_2$, and adding an edge $e$ directed from $v_1$ to $v_2$. If $v_1$ is a sink in $\overline{G'}\setminus e$ \textbf{or} $v_2$ is a source in $\overline{G'}\setminus e$, we will say $\overline{G'}$ is obtained from $\overline{G}$ by a \emph{consistent vertex expansion}.   Similarly, if $\overline{G}$ is obtained from $\overline{G'}$ by contracting such an edge, we will say $\overline{G}$ is obtained by a \emph{consistent edge contraction}.  See Figure \ref{minor_example_fig}.

\begin{figure}
\includegraphics[scale=.65]{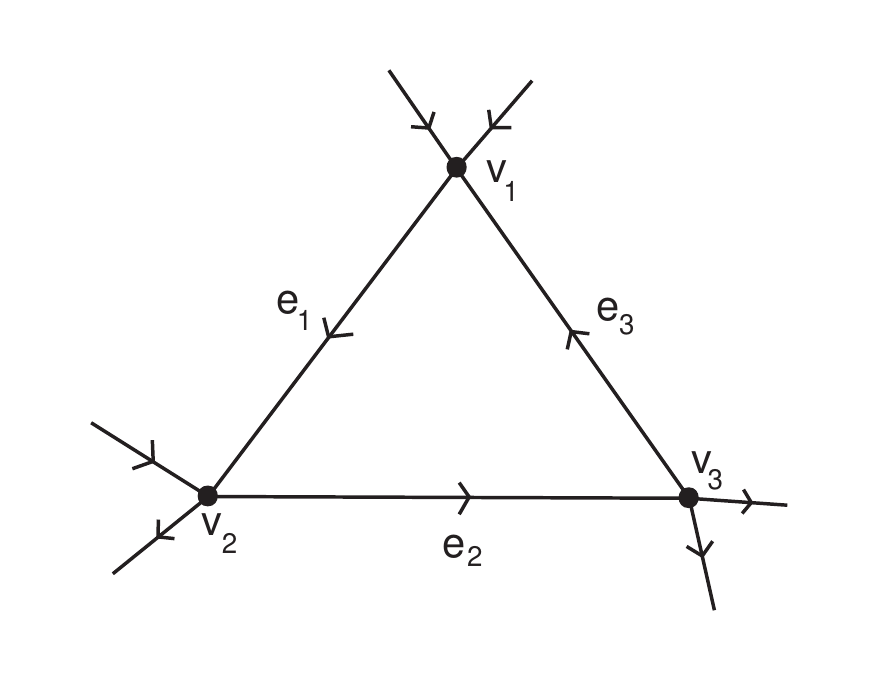}
\caption{The edges $e_1$ and $e_2$ are candidates for consistent edge contraction, but edge $e_3$ is not.  Note also that the vertices $v_1, v_2, v_3$ can be replaced by a single vertex using H-cyclic subcontraction, but no subset of them can be H-cyclicly subcontracted.}
\label{minor_example_fig}
\end{figure}

\begin{prop}
Let $\overline{G}$ and $\overline{G'}$ be directed graphs. If $\overline{G}$ is intrinsically linked and $\overline{G'}$ is obtained from $\overline{G}$ by a consistent vertex expansion, then $\overline{G'}$ is instrinsically linked as a digraph.   Equivalently, if $\overline{G'}$ has a linkless embedding, and $\overline{G}$ is obtained from $\overline{G'}$ by a consistent edge contraction, then $\overline{G}$ has a linkless embedding.  That is, the property of having a linkless embedding is closed under consistent edge contraction.
\end{prop}

\begin{proof}
There is a one-to-one correspondence of consistently oriented cycles in $\overline{G}$ and $\overline{G'}$ under these moves, so we may proceed as in the undirected case.

In any embedding $f'$ of $\overline{G'}$, we may shrink the edge $e$ by isotopy until it is contained a ball disjoint from the rest of $\overline{G'}$.  This gives an embedding $f$ of $\overline{G}$.  If $f$ contains a consistently oriented non-split link, then isotopic consistently oriented cycles can be found in $f'$, so it contains a link as well.
\end{proof}

We note that the operation of taking minors of directed graphs described by \cite{FHR} corresponds to the \emph{weak subcontraction} of Jagger \cite{jag}.  Jagger also defines \emph{strong subcontraction} as a generalization of undirected graph minors to directed graphs.  If $\overline{H}$ is obtained from $\overline{G}$ by strong subcontraction, any consistently oriented cycle of $\overline{H}$ corresponds to at least one consistently oriented cycle of $\overline{G}$.   Hence, strong subcontraction avoids the problems highlighted by Theorem \ref{vertex_exp_thm}, as neither $\overline{G'}$ nor $\overline{G''}$ strongly subcontracts to $\overline{G}$.  Strong subcontraction is defined combinatorially and generally does not have a clear geometric interpretation, but it may be a promising technique for studying intrinsically linked directed graphs. We will define an operation we call \emph{H-cyclic subcontraction} that is more restrictive than strong subcontraction, but easier to understand geometrically.

First we review some definitions.  Recall that a digraph $\overline{G}$ is strongly connected if for any pair of vertices $v,w$ in $\overline{G}$, there exists a directed path from $v$ to $w$. Let $\overline{H}$ be a digraph with vertex set $\{ w_0, w_1 \ldots w_r\}$ and $V_0, V_1, \ldots V_r$ a partition of the vertices of $\overline{G}$ such that if $w_iw_j$ is an edge of $\overline{H}$ there exist $v_{ik} \in V_i$ and $v_{jl} \in V_j$ such that $v_{ik}v_{jl}$ is an edge of $\overline{G}$.  Then $\overline{H}$ is a minor of $\overline{G}$ if the induced subgraph $G(V_i)$ is connected as an undirected graph for all $V_i$ (that is, weakly connected as a digraph).  The graph $\overline{H}$ is a strong subcontraction (or strong minor) of $\overline{G}$ if the induced subgraph $\overline{G(V_i)}$ is strongly connected as a digraph for all $V_i$.  We define $\overline{H}$ to be an H-cyclic minor (or H-cyclic subcontraction) if the induced subgraph $\overline{G(V_i)}$ contains a consistently oriented Hamiltonian cycle for all $V_i$.  Note that contracting a symmetric pair of edges is a special case of H-cyclic subcontraction.

\begin{prop}
Let $\overline{H}$ be an H-cyclic minor of a digraph $\overline{G}$. Then if $\overline{G}$ has an embedding where no pair of consistently oriented cycles have non-zero linking number, then $\overline{H}$ has an embedding where no pair of consistently oriented cycles have non-zero linking number.
\label{almost_closed_prop}
\end{prop}

\begin{proof}
If there is more than one partition $V_i$ of $\overline{G}$ that contains more than one vertex, we may proceed sequentially. Thus, we can assume $V_0 = \{v_{01} \ldots v_{0n}\}$ and $V_i = v_{i0}$ for all $i \neq 0$.  Let $f$ be an embedding of $\overline{G}$ with no link with non-zero linking number.  Let $C_0$ be the consistently oriented Hamiltonian cycle of $\overline{G(V_0)}$.   We form an embedding of $\overline{H}$ as follows:  in $f(G)$, delete one edge of $C_0$ to form a consistent path $P_0$ that includes all vertices of $V_0$.  Delete all edges in $\overline{G(V_0)}\setminus P_0$.  Contract $P_0$.   This gives an embedding $f$ of $\overline{H}$.

If $f(\overline{H})$ has a link $L_1, L_2$ with nonzero linking number that does not use $w_0$, then this link is contained in $f(\overline{G})$, a contradiction.  So we may assume that $w_0 \in L_1$.  By construction, there are vertices $v_{0i}$ and $v_{0j}$ such that $L_1$ corresponds to a path $P_1$ in $f(\overline{G} \setminus \overline{G(V_0)})$ that connects $v_{0i}$ and $v_{0j}$.  Let $P_{ij}$ be the subpath of $P_0$ that connects $v_{0i}$ and $v_{0j}$.  Then $P_1 \cup P_{ij}$ in $f(\overline{G})$ is isotopic to $L_1$.  If $P_1 \cup P_{ij}$ is consistently oriented, then $lk(P_1 \cup P_{ij}, L_2) \neq 0$, a contradiction.

If $P_1 \cup P_{ij}$ is not consistently oriented, we may form $P_{ji} = C_0 \setminus P_{ij}$.  Then $P_1 \cup P_{ji}$ is consistently oriented, as is $C_0$.  Since $C_0, P_1 \cup P_{ji}$ and $P_1 \cup P_{ij}$ divide $S^2$ into regions, and $L_2$ has non-zero linking number with $P_1 \cup P_{ij}$, $L_2$ must have non-zero linking number with either $C_0$ or $P_1 \cup P_{ji}$.   This gives a consistently oriented link with non-zero linking number in $f(\overline{G})$, a contradiction.
\end{proof}

Note that every undirected graph that is intrinsically linked contains a link with non-zero linking number, due to the classification of forbidden minors for intrinsically linked graphs \cite{rst}.  If every embedding of an intrinsically linked digraph must contain a link with non-zero linking number, then Proposition \ref{almost_closed_prop} shows that the property of having a linkless embedding is closed under H-cyclic subcontraction.  A similar caveat is necessary for the result claimed in Theorem 4.3 of \cite{FHR}.

\section{An Intrinsically Knotted Directed Graph}
\label{i_knot_section}

The following is a direct corollary of a result of Taniyama-Yasuhara \cite{ty}, and independently Foisy \cite{foisy}, applied to the directed graph case.  Let $D_4$ be the graph pictured in Figure \ref{D4Bar_pic}.

\begin{cor} Let $\overline{D_4}$ be an oriented $D_4$ graph, with all edges oriented in the counterclockwise direction, or all edges oriented in the clockwise direction.  If $lk(C_1,C_3) \neq 0$ modulo 2 and $lk(C_2,C_4) \neq 0$ modulo 2, then $\overline{D_4}$ contains a consistently oriented Hamiltonian cycle that is a non-trivial knot.
\label{D4_cor}
\end{cor}

\begin{proof}
Ignoring edge orientations, the Taniyama-Yasuhara/Foisy result implies that there is a Hamiltonian cycle in $\overline{D_4}$ with odd Arf invariant.   As all Hamiltonian cycles of $\overline{D_4}$ are consistently oriented, we have the result.
\end{proof}

\begin{figure}
\includegraphics[scale=.35]{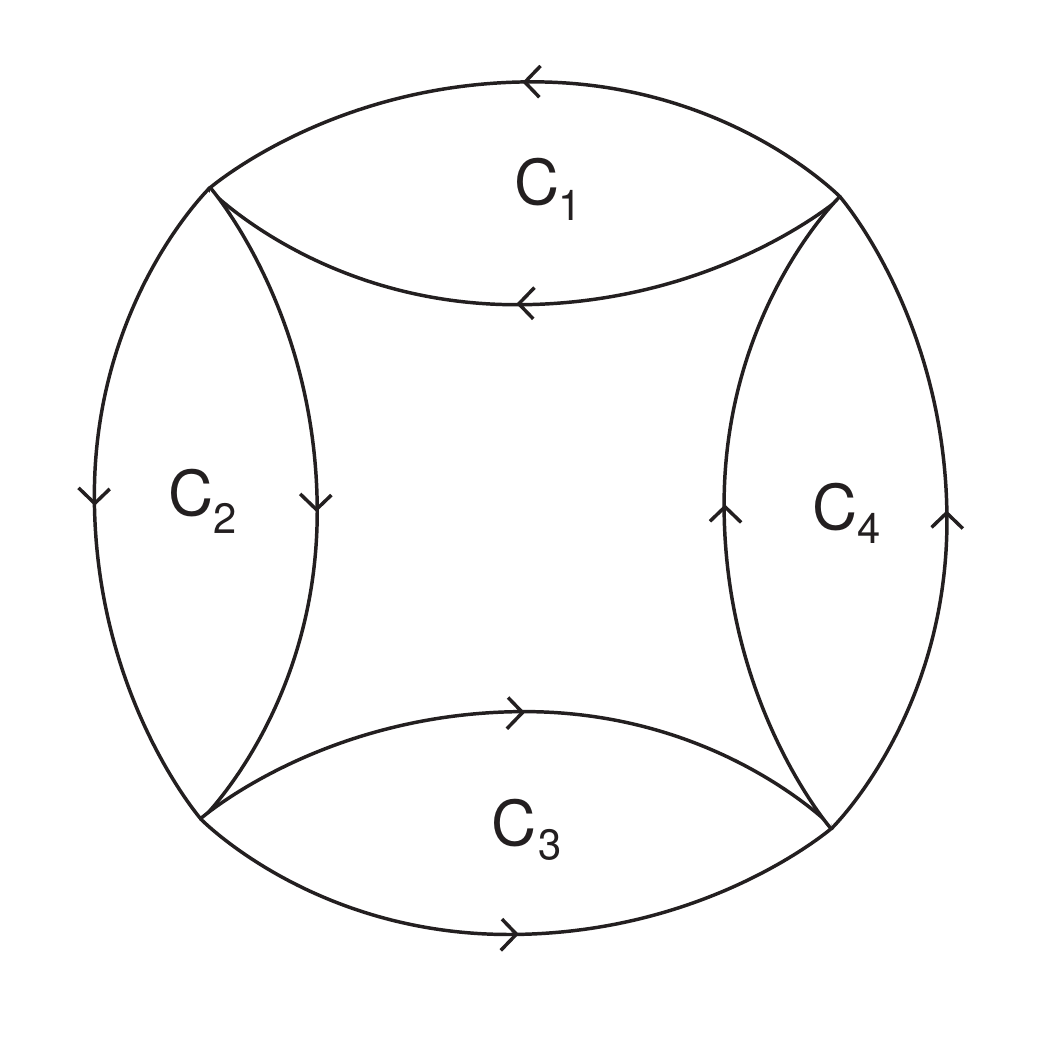}
\caption{The counterclockwise oriented $\overline{D_4}$.}
\label{D4Bar_pic}
\end{figure}

If every embedding of a digraph $\overline{G}$ can be reduced to a copy of $\overline{D_4}$ via edge deletions and consistent edge contractions (in the sense of Section \ref{consistent_section}), and this copy of $\overline{D_4}$ satisfies the conditions of Corollary \ref{D4_cor}, then every embedding of $\overline{G}$ contains a consistently oriented cycle that is a non-trivial knot.  That is, $\overline{G}$ is intrinsically knotted as a digraph.

We now produce an example of a consistently oriented intrinsically knotted digraph $\overline{G}$ on 11 vertices.  We construct $\overline{G}$ from two vertex sets, $A$ of five vertices ($a_1, a_2, a_3, a_4, a_5$) and $B$ of four ($b_1, b_2, b_3, b_4$), and 2 distinguished vertices $v$ and $w$ as follows.  Add edges so that $a_1, a_2, a_3,b_1,b_2,b_3$ form a $K_{3,3}$ with all edges directed from $a_i$ to $b_j$.  Add edges from $a_1, a_2,a_3$ to $v$ and edges from $v$ to $b_1,b_2,b_3$.  Add a 3 cycle of arbitrary orientation between $b_1,b_2,b_3$.  Add 3 edges, directed from each $b_i$ to $b_4$.  Add edges from each $b_i$ to each $a_j$.  Add an edge from each $b_i$ to $w$ and from $w$ to each $a_j$.  Add edges so that $a_1, a_2, a_3$ form a $\overline{DK_3}$.  Add an edge between $a_4,a_5$ and from each of $a_4$ and $a_5$ to each of $a_1, a_2, a_3$.

\begin{thm}
Every embedding of $\overline{G}$ contains a consistently oriented cycle that is a non-trivial knot.
\end{thm}

\begin{proof}

Within any embedding of $\overline{G}$ we can identify a copy of $\overline{D_4}$ that satisfies the conditions of Corollary \ref{D4_cor} as follows.
The vertices $a_1, a_2, a_3, b_1,b_2,b_3$ and $v$ with the edges directed from $A$ to $B$ form $K_{3,3,1}$.  We may find a triangle-square pair with odd linking number.  Suppose the triangle $T_1$ is $a_1, v, b_1$.  The square is $a_2,b_2, a_3, b_3$.  Using the edge $b_2b_3$ we find a triangle $T_3$ that has odd linking number with $T_1$.  $T_3$ contains $b_2, b_3$ and one of $a_2, a_3$.  Say $a_2$.  Note that $a_1$ and $a_2$ are sources. The vertex $b_1$ is a sink, and so is one of $b_2, b_3$.  Say it is $b_2$.

Now, $b_1, b_2, b_4, a_3, a_4, a_5$ and $w$ along with the edges directed from $B$ to $A$ form a $K_{3,3,1}$. (This does not depend on the choice of vertex labels in the previous paragraph).  We may similarly find two triangles $T_2$ and $T_4$ with odd linking number.

We now construct the $\overline{D_4}$. If the source vertices of $T_2$ and $T_4$ are $b_1$ and $b_2$, they are the sink vertices of $T_1$ and $T_3$ and we continue.   If $T_4$ has source vertex $b_4$ instead of $b_2$ we add the edge $b_2b_4$.  $T_2$ and $T_4$ have sink vertices in $\{a_3, a_4, a_5\}$.  These are disjoint from the source vertices of $T_1$ and $T_3$, ($a_1$ and $a_2$).   As each of $a_3, a_4, a_5$ has an edge directed to each of $a_1$ and $a_2$, we can add the edges needed to get the adjacencies required to complete the $\overline{D_4}$.

Note that the other vertices contained in $T_1, T_3$ ($v$ and $b_2$) are disjoint from $T_2, T_4$, and the other vertices of $T_2, T_4$ ($w$ and one of $a_3, a_4, a_5$) are disjoint from $T_1, T_3$.
\end{proof}

\section{Constructing a 3-linked directed graph}
\label{3_link_section}

We include this construction of an intrinsically 3-linked directed graph both to demonstrate key ideas for the 4-linked construction and to provide a bound for Question 5.5 of \cite{FHR}.  Question 5.5 asks for the minimal $n$ such that the symmetric complete digraph $\overline{DK_n}$ is intrinsically 3-linked.  The example below shows that $n \leq 21$. As $K_{9}$ can be embedded with no triple link \cite{flapan3}, and any embedding of $K_n$ with no 3-link can be extended to an embedding of $\overline{DK_n}$ with no 3-link (by embedding the symmetric edges so that they bound disks), we have $10 \leq n \leq 21$.

\begin{lemma}
Let $\overline{H}$ be $\overline{DK_{4,4}}$ with a single directed edge $b-a$ removed. Then the edge $a-b$ is contained in a consistently oriented link with non-zero linking number in every embedding of $\overline{H}$.  Further, that link is composed either of two 4-cycles, or a 4-cycle containing $a-b$ and a 2-cycle.
\label{graph_H_lma}
\end{lemma}
\begin{proof}
Let $\overline{H'}$ be the subgraph of $\overline{H}$ shown in Figure \ref{graph_H}. The underlying graph of $H'$ is $K_{4,4}$ with the vertex partitions $\{b,c,e,g\}$, and $\{a,d,f,h\}$, and every edge between $\{d,f,h\}$ and $\{c,e,g\}$ directed towards $\{c,e,g\}$, every edge from $b$ to $\{d,f,h\}$ directed towards $\{d,f,h\}$, and every edge between $\{c,e,g\}$ and $a$ directed towards $a$, and finally the edge between $a$ and $b$ directed towards $b$. Sachs \cite{sachs} showed that given any edge of $K_{4,4}$ and any spatial embedding of $K_{4,4}$, that edge is contained in a non-split link. By construction of $\overline{H'}$, the edge $a-b$ is only contained in consistently oriented cycles, so in any embedding of $\overline{H}$, there must exist a consistently oriented 4-cycle $S_1$ containing $a-b$ that has non-zero linking number with an arbitrary 4-cycle $S_2$.

Fix an embedding of $\overline{H}$.  We may find cycles $S_1, S_2$ as above.  If $S_2$ is consistently oriented, we are done.   If $S_1$ has non-zero linking number with a 2-cycle $w-z, z-w$ for two adjacent vertices $w,z$ in $S_2$, then we have the result.  So we may assume that $S_1$ has zero linking number with all such 2-cycles.   In this case, we may form $S'_2$ from $S_2$ by replacing edge $z-w$ with $w-z$ and as $lk(S_1, z-w,w-z) = 0$, $lk(S_1, S_2) = lk(S_1,S'_2)$.  Thus, if $S_2$ is inconsistently oriented, we may form a consistently oriented 4-cycle $S'_2$ by replacing one or more directed edges of $S_2$ with edges of opposite orientation, and as $lk(S_1, S'_2) = lk(S_1,S_2) \neq 0$, we have the result.

\end{proof}

\begin{figure}
\includegraphics[scale=.25]{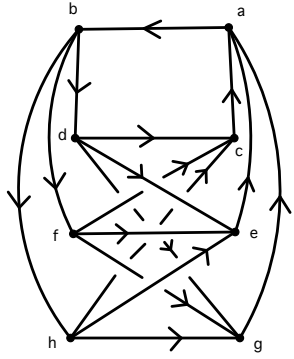}
\caption{The digraph $\overline{H'}$.}
\label{graph_H}
\end{figure}

\begin{theorem}
Let $\overline{H}$ be as above, and $\overline{G}$ the graph formed from 3 copies of $\overline{H}$ by identifying the end points of the preferred edges $a-b$ so that they form a consistently oriented triangle.  Then every embedding of $\overline{G}$ contains a non-split consistently oriented 3-link.
\end{theorem}

\begin{proof}

The graph $\overline{G}$ is formed from three copies of $\overline{H}$, label them $\overline{H_i}$.  Embed $\overline{G}$.  Let $A_i$ be the cycles containing $a_i$ and $b_i$ in the non-split, consistently oriented link in $\overline{H_i}$.   Let $C_i$ be the cycle in $\overline{H_i}$ that has non-zero linking number with $A_i$. See Figure \ref{3construction}.  Let $Z$ be the consistent 3-cycle formed from the edges $a_i-b_i$, and let $W$ be the consistent 9-cycle formed from the other edges of the $A_i$.   (Z is the ``inside" of the ring, W is the ``outside").  Then $A_i, Z, W$ are a division of $S^2$ into 5 consistent cycles, and hence each of the $C_i$ must have non-zero linking number with at least two of these cycles.  By the pigeonhole principle, two of the $C_i$ have non-zero linking number with the same cycle, giving a 3-link.

\end{proof}

\begin{figure}
\includegraphics[scale=.75]{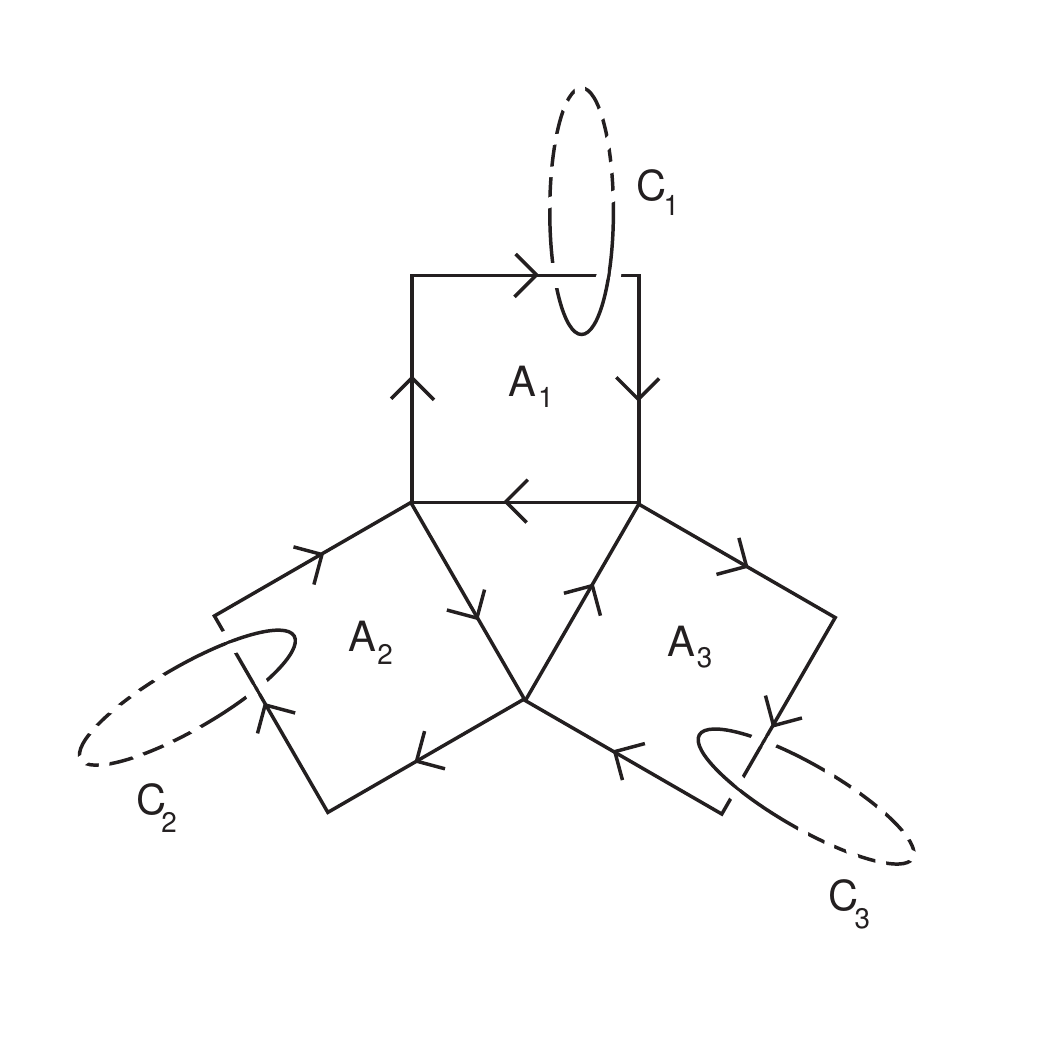}
\caption{Gluing three copies of $\overline{H}$.}
\label{3construction}
\end{figure}

\section{Constructing a 4-linked directed graph}
\label{4_link_section}

Let $\overline{H}$ be the digraph in Figure \ref{graph_H}.  We will use the digraph $\overline{H}$ as the building block for the construction of the 4-linked directed graph, and Lemma \ref{graph_H_lma} will be a key tool in the proof.

\begin{lemma} Let $\overline{G}$ be the digraph formed from 11 copies of $\overline{H}$, by identifying $a_1$ with $b_2$, $a_2$ with $b_3$ and so on to $a_{11}$ identified to $b_1$.  Then every embedding of $\overline{G}$ contains either a consistently oriented non-split 4-link or a set of three non-split 3-links that have no common cycle.
\label{3linklemma}
\end{lemma}

\begin{proof}
Embed $\overline{G}$, and assume that the embedding does not contain a consistently oriented non-split 4-link.   Let $A_i$ be the cycles containing $a_i$ and $b_i$ in the non-split, consistently oriented link in $\overline{H_i}$, and let $C_i$ be the other components in those links.   Let $Z$ be the consistent 11-cycle formed from the edges $a_i-b_i$, and let $W$ be the consistent 33-cycle formed from the other edges of the $A_i$.   (Z is the ``inside" of the ring, W is the ``outside").  Then $A_i, Z, W$ are a division of $S^2$ into 13 consistent cycles, and thus each of the 11 $C_i$ must have non-zero linking number with at least 2 of these cycles.

As there is no 4-link, at most 2 of the $C_i$ can have non-zero linking number with $Z$, and at most 2 can have non-zero linking number with $W$. Thus at least 7 of the $C_i$ must have non-zero linking number with $A_j$ where $i\neq j$.  As there is no 4-link, for each $A_i$ there is at most one $C_j$, $i \neq j$ that has non-zero linking number with $A_i$.   Thus the set $A = \{ A_i | lk(A_i, C_j) \neq 0$ for $i \neq j \}$ has at least 7 members.  Choose a member of $A$, call it $A_1$.   It has non-zero linking number with $C_1$ from $\overline{H_1}$ and a second $C_j$, call it $C_2$.   Remove $C_{1}$ and $C_{2}$ from consideration.  This leaves at least 5 $C_i$ that have non-zero linking number with $A_j$ where $i\neq j$.  Thus $A$ must have at least 5 members remaining, and we may choose $A_3$ in $A$, with $A_3 \neq A_2$.  $A_3$ has non-zero linking number with $C_3$ from $\overline{H_3}$ and an additional cycle $C_4$.  As before, removing $C_3$ and $C_4$ from consideration leaves at least 3 of the $C_i$ with non-zero linking number with $A_j$ where $i\neq j$.  Thus $A$ must have at least 3 members, and we may choose $A_5 \in A$ with $A_5 \neq A_2, A_4$. The cycle $A_5$ has non-zero linking number with $C_5$ and $C_6$. So we have three 3-links $A_1$, $C_{1}$, $C_{2}$, $A_3$, $C_{3}$, $C_{4}$, and $A_5$, $C_{5}$, $C_{6}$,  with all of the $C_{j}$, $A_i$ distinct (though the $A_i$ may share vertices).

\end{proof}

\begin{theorem} Let $\overline{H}$ be as above, and $\overline{G}$ the digraph formed from 2 ${ 11 \choose 2}$ - 11 copies of $\overline{H}$ by identifying the $a_i$ and $b_j$ such that the edges $a_i b_i$ form $\overline{DK_{11}} \setminus B$, where $B$ is a consistently oriented Hamiltonian cycle.  Then $\overline{G}$ contains a non-split consistently oriented 4 component link in any spatial embedding.
\end{theorem}

\begin{proof}
Embed $\overline{G}$, and let $\overline{G'}$ denote the subgraph of $\overline{G}$ made up of the 11 copies of $\overline{H}$ whose edges $a_i - b_i$ make up the unique consistently oriented Hamiltonian cycle of $\overline{G}$ that corresponds to the missing Hamiltonian cycle $B$ (with opposite orientation). If $\overline{G'}$ does not contain a 4-link, then by Lemma \ref{3linklemma}, it must contain three distinct 3-links.   Renumbering the $\overline{H_i}$ as necessary, we may call the components of these links $A_2, C_{21}, C_{22}$; $A_4, C_{41}, C_{42}$, and $A_6, C_{61}, C_{62}$, where $A_i$ is the component containing the edge $a_i-b_i$, and $C_{ij}$ the cycles that have non-zero linking number with $A_i$.

If $A_2$, $A_4$ and $A_6$ do not share any vertices, we use the copies of $\overline{H}$ that make up the rest of $\overline{G}$ to form a new subgraph $\overline{G''}$ of $\overline{G}$ as follows.  Let $\overline{H_1}$ be the copy of $\overline{H}$ where $a_1$ is identified with $b_6$ and $b_1$ with $a_2$.  $\overline{H_3}$ is the graph with $a_3$ identified with $b_2$ and $b_3$ with $a_4$, and similarly for $\overline{H_5}$.   Call the components of the 2-links in $\overline{H_i}$: $A_i$ and $C_{i1}$ (for $i$ odd).

Then we may form a consistent 6-cycle $Z$ from the edges $a_i-b_i$ of $\overline{G''}$, and an 18-cycle $W$ from the other edges of the $A_i$.  The cycles $A_i$, $Z$ and $W$ divide $S^2$ into consistently oriented regions, so each $C_{ij}$ must have non-zero linking number with at least 2 cycles in $A = \{A_i, Z, W\}$.  There are 9 $C_{ij}$ and 8 cycles in $A$.   Each $C_{ij}$ must have non-zero linking number with at least 2 cycles in $A$, and as $\frac{2*9}{8} > 2$, some cycle in $A$ must have non-zero linking number with at least 3 of the $C_{ij}$.   This gives the 4-link. (Alternately, we can apply Lemma \ref{genlemma} below, and as 6 $>$ 2(3-1), $\overline{G''}$ must contain a 4-link.)

If the some of $A_2$, $A_4$, $A_6$ share a vertex, we may repeat the above argument omitting one (or more) of the ``bridging" copies of $\overline{H}$ in the construction of $\overline{G''}$.
\end{proof}

\section{Towards an n-linked directed graph}

We expect the construction of an intrinsically $n$-linked directed graph for $n>4$ to be possible, and include the following lemmas.   Each may serve as the final step in an inductive argument constructing such a graph.

\begin{lma}
If $c_1$ is a consistently oriented cycle in $\overline{DG}$ and $c_1$ has non-zero linking number with $n$ disjoint arbitrary cycles, then $c_1$ has non-zero linking number with $n$ disjoint consistently oriented cycles.
\end{lma}

\begin{proof}
If $lk(c_1,c_j) \neq 0$ then either $lk(c_1, \overline{v_iv_j} ~ \overline{v_jv_i}) \neq 0$ for some adjacent vertices $v_i, v_j$, in $c_j$, or $c_j$ can be modified to a consistent cycle $c'_j$ by replacing any inconsistent edge $\overline{v_iv_j}$ with the consistent edge $\overline{v_jv_i}$ without changing $lk(c_1, c_j)$.  As only the cycle $c_j$ is modified, $lk(c_1,c_i)$ is unaffected for $i \neq j$.
\end{proof}

\begin{lemma}
Let $\overline{H}$ be a digraph such that in every embedding of $\overline{H}$, a fixed edge is contained a consistent cycle that has non-zero linking number with $n-1$ disjoint, consistently oriented cycles of $\overline{H}$. Let $\overline{H'}$ be a digraph for which every embedding contains a 2 link with non-zero linking number, and that link uses a fixed edge. Let $k$ be even, and $\overline{G}$ a graph formed from $\frac{k}{2}$ copies of $\overline{H}$ and $\frac{k}{2}$ copies of $\overline{H'}$ so that the preferred edges form a consistent $k$ cycle.  Then every embedding of $\overline{G}$ contains a consistently oriented non-split $n+1$ link when $k > 2(n-1)$.
\label{genlemma}
\end{lemma}

\begin{proof}
Let the components of the links in $\overline{H}$ and $\overline{H'}$ that use the preferred edge be denoted as $A_i$, and the other components of those links as $C_{ij}$.  Let $Z$ be the consistently oriented $k$ cycle formed by the preferred edges, and $W$ the (consistent) cycle formed by the other edges of the $A_i$.  Then $A_i$, $Z$, $W$ divide $S^2$ into $k+2$ consistently oriented regions.

There are $\frac{k}{2} n$ cycles $C_{ij}$, and each of these must have non-zero linking number with at least two regions.  Thus, the $C_{ij}$ have non-zero linking number with at least $kn$ regions in total.

Some region has non-zero linking number with at least $n$ of the $C_{ij}$ when $\frac{kn}{k+2} > n-1$.  Thus $\overline{G}$ must contain an $n+1$ link when $k>2(n-1)$.
\end{proof}

\section{Counting links and knots in directed graphs}

Recall that for a graph $G$, the digraph $\overline{DG}$ is the symmetric digraph corresponding to $G$ formed by replacing each edge $v_1 v_2$ of $G$ with the directed edges $v_1v_2$ and $v_2v_1$.
In \cite{FHR}, Question 5.2 asks if $G$ is intinsically linked, what is the minimum number of distinct links in any embedding of $\overline{DG}$?  We provide some general bounds and a precise answer for the case when $G$ has an embedding with a single non-trivial link.

We will follow notation from \cite{fmellor} and \cite{mellor2} and let $mnl(G)$ denote the minimal number of distint, non-split links (of any number of components) over all embeddings of $G$, and $\overline{mnl}(\overline{G})$ the minimal number of distinct, consistently oriented nonsplit links over all embeddings of a directed graph $\overline{G}$. Let $mnl_n(G)$ denote the minimal number of distinct non-split $n$ component links in any embedding of a graph $G$, and $\overline{mnl_n}(\overline{G})$ denote the minimal number of distinct consistently oriented non-split $n$ component links in any embedding of a directed graph $\overline{G}$.  Let $mnk(G)$ denote the minimum number of nontrivial knots over all embeddings of $G$, and $\overline{mnk}(\overline{G})$ be the minimum number of consistently oriented nontrivial knots in any embedding of a directed graph $\overline{G}$.  Let $\Gamma(G)$ denote the set of all cycles in $G$.

\begin{prop}
For a graph $G$, $\overline{mnk}(\overline{DG}) \leq 2*mnk(G)$.
\end{prop}

\begin{proof}
Let $f$ be an embedding of $G$ that realizes $mnk(G)$. We may extend $f$ to an embedding of $\overline{DG}$ by thickening the edges of $f(G)$ and embedding $\overline{DG}$ so that the bigons formed by the edges corresponding to $e$ bound a disk within that thickened edge.  Then for a knot $K$ in $f(G)$ there are two isotopic, consistently oriented cycles, $K^+$ and $K^-$, in $f(\overline{DG})$.  Any other consistently oriented cycle $C$ in $f(\overline{DG})$ must be a trivial knot, as if $C \in \Gamma(f(G))$ when forgetting edge orientations, then $C$ is trivial.  If $C \notin \Gamma(f(G))$ when forgetting edge orientations, then $C$ is a bigon corresponding to an edge $e$, and hence bounds a disk by the construction of $f(\overline{DG})$.   Thus $f(\overline{DG})$ contains exactly $2*mnk(G)$ nontrivial knots, and we have the bound.
\end{proof}

\begin{lemma}
Let $f(G)$ be an embedding of a spatial graph $G$ that contains $k_n$ non-split $n$ component links.  Then $\overline{mnl_n}(\overline{DG}) \leq k_n*2^n$ and $\overline{mnl}(\overline{DG}) \leq \sum_n k_n*2^n$.
\label{directed_link_count}
\end{lemma}

\begin{proof}
We may extend $f(G)$ to an embedding of $\overline{DG}$ by thickening the edges of $f(G)$ and then embedding $\overline{DG}$ so that the bigon formed by the directed  edges corresponding to an edge $e$ bound a disk within the thickened edge $f(e)$. Call this embedding $f(\overline{DG})$.

We now count the distinct non-split $n$ component links in $f(\overline{DG})$.  Given a nonsplit link $L_1 \ldots L_n$ in $f(G)$, we choose an orientation for each component.  Because all of the bigons formed by birected edges bound disks in $f(\overline{DG})$, the cycles $L_i^+$ and $L_i^-$ from $f(\overline{DG})$ may be thought of an element  of $f(\Gamma(G))$, specifically, they are isotopic to $L_i$. Thus for any choice of orientation, $L_1^{\pm} \ldots L_n^{\pm}$ is a distinct, non-split, consistently oriented link in $f(\overline{DG})$. Thus $f(\overline{DG})$ contains at least $k_n*2^n$ such links.

Suppose that $L_1 \ldots L_n$ is a consistently oriented $n$ component link not constructed in the manner of the preceeding paragraph.  Suppose every $L_i$ can be mapped to an element $L'_i$ of $f(\Gamma(G))$ by forgetting orientation.   Then the link $L'_1 \ldots L'_n$ is split, so $L_1 \ldots L_n$ is split as well.   Thus, some $L_i$ does not map to a cycle in $f(\Gamma(G))$. Then $L_i$ is a bigon in $\Gamma(\overline{DG})$ formed from doubling an edge of $G$, and hence $L_i$ bounds a disk in $f(\overline{DG})$.  This implies that $L_1 \ldots L_n$ is split as well.

Thus, the embedding $f(\overline{DG})$ has exactly $k_n*2^n$ distinct consistently oriented non-split $n$ component links, giving the bounds.
\end{proof}

\begin{cor}
For a graph $G$, $\overline{mnl_n}(\overline{DG}) \leq mnl_n(G)*2^n$.
\end{cor}

\begin{proof}
Choose $f(G)$ to be an embedding that realizes $mnl_n(G)$, and apply Lemma \ref{directed_link_count}.
\end{proof}

\begin{cor}
If a graph $G$ has an embedding that simultaneously realizes $mnl_n(G)$ for all $n$, then $\overline{mnl}(\overline{DG}) \leq \sum_n mnl_n(G)*2^n$.
\label{count_bound_cor}
\end{cor}

\begin{proof}
Choose $f(G)$ to be an embedding that simultaneously realizes $mnl_n(G)$, and apply Lemma \ref{directed_link_count}.
\end{proof}

While we do not have an example, it seems that the condition on the simultaneous realization of $mnl_n(G)$ should not be vacuous.  In \cite{foisy2}, the second author found a graph $G$ with $mnk(G)=0$ and $mnl_3(G)=0$, but such that every embedding of $G$ contains either a 3-link or a nontrivial knot.  Hence, the lower bounds of $mnk(G)$ and $mnl_3(G)$ cannot be simultaneously realized.  Similar behavior is known for several of the graphs in Heawood family \cite{nikk2}.  Thus, it seems possible that there is a graph $G'$ such that $mnl_n(G')$ and $mnl_m(G')$ cannot be simultaneously realized in any embedding of $G'$.

\begin{prop}
Let $G$ be a graph with $mnl(G) = 1$.  Then $\overline{mnl}(\overline{DG}) = 4$.
\label{count_lemma}
\end{prop}

\begin{proof}
As $G$ is intrinsically linked, it must contain one of the Petersen family graphs as a minor \cite{rst}.  Thus, every embedding of $G$ must contain a non-split 2 component link with non-zero linking number.  So, $mnl_2(G) = 1$, and as $mnl(G) = 1$, $mnl_n(G) = 0$ for all $n > 2$. Note that an embedding $f(G)$ that realizes $mnl(G)$ also realizes $mnl_n(G)$ simultaneously for all $n$.  Thus by Corollary \ref{count_bound_cor} $\overline{mnl}(\overline{DG}) \leq 4$.

Let $f(\overline{DG})$ be an arbitrary embedding of $\overline{DG}$.  We may obtain an embedding of $G$ by deleting one of each pair of directed edges and ignoring orientations.  Thus, in $f(\overline{DG})$ we may find two cycles $L_1$ and $L_2$ (possibly with inconsistent orientation) that have non-zero linking number.

Suppose $L_1$ is $k$ edges in length.  The edges of $L_1$ and the edges with opposite orientation can be thought of as subdividing the sphere into $k+2$ regions, where the boundary of each region is a consistently oriented cycle ($k$ bigons and 2 $k$-gons).  See Figure \ref{link_count_example_fig}.  As $L_2$ has non-zero linking number with $L_1$, it must have non-zero linking number with at least two of these consistently oriented cycles, call them $\alpha_1$ and $\alpha_2$.

\begin{figure}
\includegraphics[scale=.5]{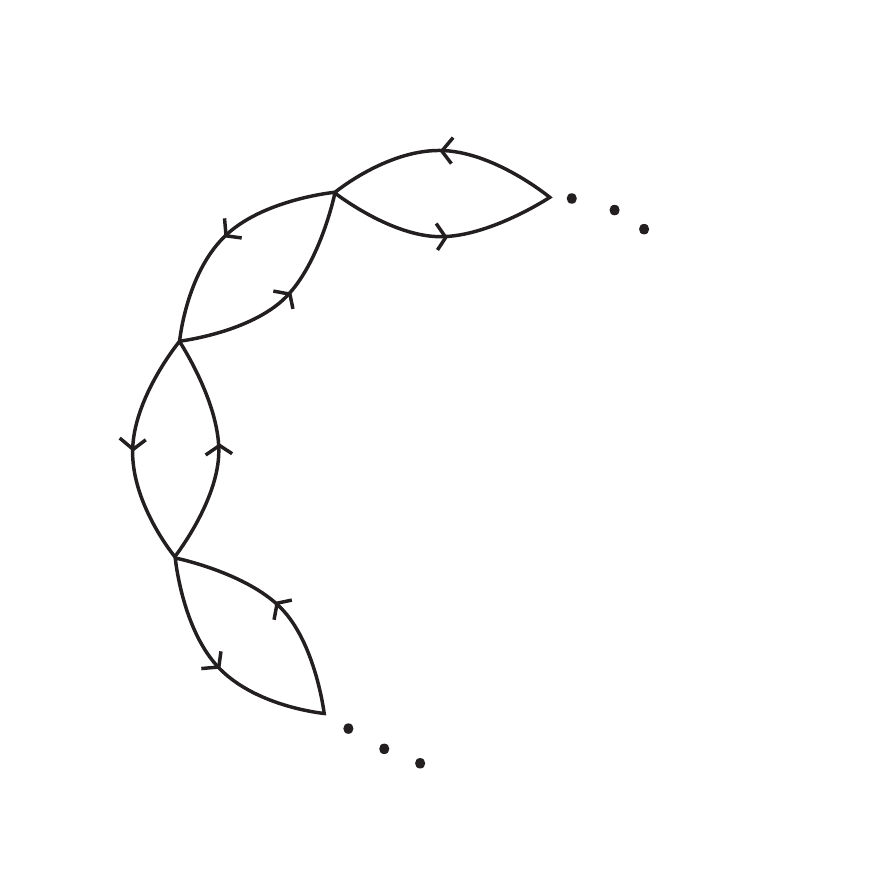}
\caption{The bigons formed from $L_1$.}
\label{link_count_example_fig}
\end{figure}

Similarly, we now consider the $k'$ edges of $L_2$ and those of opposite orientation.   They divide the sphere into $k'+2$ regions whose boundaries are consistently oriented cycles of $\overline{DG}$, so each $\alpha_i$ must have nonzero linking number with two consistent cycles $\beta_{i1}$ and $\beta_{i2}$.   This gives four distinct consistently oriented nonsplit links in $f(\overline{DG})$:  $(\alpha_1, \beta_{11}), (\alpha_1, \beta_{12}), (\alpha_2, \beta_{21})$ and $(\alpha_2, \beta_{22})$.

Thus, every embedding of $\overline{DG}$ must contain at least $4$ distinct, consistently oriented non-split links.
\end{proof}

Foisy-Howards-Rich show that if G has a linkless embedding, then $\overline{DG}$ does as well in Theroem 3.6 of \cite{FHR}.  Combining that with Proposition \ref{count_lemma}, we have the following corollary.

\begin{cor}
If $\overline{DG}$ is intrinsically linked, then $\overline{mnl}(\overline{DG}) \geq 4$.
\end{cor}

Note that this bound need not hold for arbitrary directed graphs.  Specifically, if $\overline{G}$ is the intrinsically linked digraph of Figure \ref{counterexample}, the embedding $f'$ of $\overline{G}$ used in the proof of Theorem \ref{vertex_exp_thm} has a single non-split consistently oriented link, hence $\overline{mnl}(\overline{G}) =1$.

\section{A remark on linkless embeddings}
\label{density_section}

In Section \ref{consistent_section}, we have already highlighted one important difference between intrinsic linking in undirected graphs and intrinsic linking in directed graphs, namely that the standard operation of taking minors preserves the property of having a linkless embedding for undirected graphs, but does not do so in the directed graph case. Here we highlight another important difference between the directed and undirected cases.  Specifically, we consider the extremal problem of the maximum number of edges for a graph $G$ or a digraph $\overline{G}$ on $n$ vertices that admits a linkless embedding.  (Alternately, for $n>6$, the minimal number of edges that forces $G$ or $\overline{G}$ to be intrinsically linked).  For undirected graphs, the following is a corollary of standard results on graph minors.

\begin{cor} The maximum number of edges in a graph $G$ on $n$ vertices that admits a linkless embedding is $O(n)$.
\end{cor}

\begin{proof}
There is a constant $c(6)$ such that if $G$ has more than $c(6)n$ edges, then $G$ has $K_6$ as a minor \cite{mader}, \cite{thom}.   Hence if $G$ has more than $c(6)n$ edges, $G$ is intrinsically linked.
\end{proof}

In contrast, for directed graphs much denser examples may not be intrinsically linked as digraphs.   Recall that a transitive tournament is an orientation of the complete graph $K_n$ so that the edge $ij$ is directed from $i$ to $j$ if $i<j$.  Theorem 3.2 of \cite{FHR} shows that a transitive tournament cannot be intrinsically linked as a digraph.  We provide an even denser example below.  Let $E(\overline{G})$ denote the edge set of $\overline{G}$ and $|E(\overline{G} )|$ the number of edges.

\begin{theorem} The maximum number of edges in a digraph on $n$ vertices that admits an embedding with no pair of disjoint, consistently oriented cycles that form a nonsplit link is $O(n^2)$.  Further, let $c$ be the constant such that $|E(\overline{G} )| > c n^2$ implies that $\overline{G}$ is intrinsically linked as a digraph. Then $ \frac{1}{2} < c \leq \frac{9}{10}.$
\label{extemal_lma}
\end{theorem}

\begin{proof}
Let $\overline{G'}$ be a transitive tournament on $n-1$ vertices.  Let $\overline{G}$ be the graph formed by adding a vertex $v$ and symmetric directed edges to and from $v$ from and to every vertex of $\overline{G'}$.  As $\overline{G'}$ contains no consistently oriented cycles, any consistently oriented cycle in $\overline{G}$ must pass through vertex $v$.  Thus $\overline{G}$ does not contain a pair of disjoint, consistently oriented cycles, so cannot be intrinsically linked as a digraph.  By construction, $\overline{G}$ has $\frac{(n+2)(n-1)}{2}$ edges.  As $\overline{DK_n}$ has $2{ n \choose 2}$ edges, the maximum number of edges in a digraph is $O(n^2)$.  Thus, the maximum number of edges in a digraph that admits a linkless embedding is $O(n^2)$.  As $\frac{(n+2)(n-1)}{2} > \frac{n^2}{2}$, the example $\overline{G}$ above shows that $c > \frac{1}{2}$.

We now address the upper bound for $c$.  Let $\overline{G}$ be a digraph with $|E(\overline{G})| > \frac{9}{10} n^2$.  Let $G'$ be an undirected graph on $n$ vertices such that $w_iw_j \in E(G')$ if $v_i$ and $v_j$ of $\overline{G}$ are connected with symmetric directed edges. As $\overline{G}$ can be formed from $\overline{DK_n}$ by deleting fewer than $\frac{1}{10} n^2$ edges, at most $\frac{1}{10} n^2$ edges of $\overline{G}$ are not part of a symmetric pair.  Thus more than $\frac{8}{10} n^2$ edges of $\overline{G}$ occur in symmetric pairs, so $G'$ has more than $\frac{4}{5} \frac{n^2}{2}$ edges.  By Tur\'{a}n's Theorem, $G'$ contains $K_6$ as a subgraph, and by the construction of $G'$, this implies $\overline{G}$ contains $\overline{DK_6}$ as a subdigraph. Therefore, $\overline{G}$ must be intrinsically linked as a digraph.
\end{proof}

A similar result holds for intrinsic knotting as $c(7)n$ edges on $n$ vertices guarantee a $K_7$ minor (\cite{mader}, \cite{thom}) and hence intrinsic knotting, but a transitive tournament with ${ n \choose 2}$ edges has no consistently oriented cycle, and hence cannot be intrinsically knotted as a digraph.

Examples of linkless or knotless digraphs with more edges than the examples discussed above may be possible.   We leave this as an open question for future research.

\medskip

\textsc{666 5th Avenue, 9th Floor,  New York, NY 10103}

\medskip

\textsc{Department of Mathematics, SUNY Potsdam,  Potsdam, NY 13676}

\end{document}